\documentclass{amsart}
\usepackage[utf8]{inputenc}

\usepackage{mathrsfs}
\usepackage{amscd,amssymb,amsopn,amsmath,amsthm,graphics,amsfonts,enumerate,verbatim,calc}

\newtheorem{theorem}{Theorem}[section]

\newtheorem{lemma}[theorem]{Lemma}

\newtheorem{corollary}[theorem]{Corollary}

\theoremstyle{definition}

\theoremstyle{remark}
\newtheorem{remark}[theorem]{Remark}

\newcommand{\bbP}{{\mathbb P}}

\newcommand{\cf}{{\rm cf}}

\newcommand{\pcf}{{\rm pcf}}
\newcommand{\ZFC}{{\rm ZFC}}
\newcommand{\tcf}{{\rm tcf}}

\newcommand{\Pcut}{{\rm Pcut}}
\newcommand{\Add}{{\rm Add}}

\newcount\skewfactor
\def\mathunderaccent#1#2 {\let\theaccent#1\skewfactor#2
\mathpalette\putaccentunder}
\def\putaccentunder#1#2{\oalign{$#1#2$\crcr\hidewidth
\vbox to.2ex{\hbox{$#1\skew\skewfactor\theaccent{}$}\vss}\hidewidth}}

\usepackage{hyperref}

\begin{document}

\title {On the number of cofinalities of cuts in ultraproducts of linear orders}

\author[M.  Golshani]{Mohammad Golshani}

\address{Mohammad Golshani, School of Mathematics, Institute for Research in Fundamental Sciences (IPM), P.O.\ Box:
	19395--5746, Tehran, Iran.}

\email{golshani.m@gmail.com}

\thanks{ The author's research has been supported by a grant from IPM (No. 1403030417). The author thanks Saharon Shelah for some useful discussions.}




\begin{abstract}
Suppose $\kappa$ is a regular cardinal and $\bar a=\langle \mu_i: i<\kappa \rangle$ is a non-decreasing sequence of regular cardinals.
We study the set of possible cofinalities of cuts $\Pcut(\bar a)=\{(\lambda_1, \lambda_2):$ for some ultrafilter $D$ on $\kappa$, $(\lambda_1, \lambda_2)$ is the cofinality of a cut of $\prod\limits_{i<\kappa} \mu_i / D      \}$.
\end{abstract}

\maketitle
\numberwithin{equation}{section}
\section{introduction}
In late 1980th, Shelah introduced the theory of possible cofinalities, abbreviated $\pcf$ and used it to prove many surprising and unexpected results in $\ZFC$, showing that $\ZFC$ is much more powerful than what was expected before by many set theorists.
Given a regular cardinal $\kappa$ and an increasing sequence of regular cardinals $\bar a=\langle \mu_i: i<\kappa \rangle$,
$\pcf(\bar a)$ is defined as
\[
\pcf(\bar a)=\{ \cf(\prod\limits_{i<\kappa}/D): D \text{~is an ultrafiler on~}\kappa       \}.
\]
One of the main results in $\pcf$ theory is that under suitable assumptions on the sequence $\bar a$, $|\pcf(\bar a)|$ is bounded and indeed $|\pcf(\bar a| < \min\{(2^\kappa)^+, \kappa^{+4}     \}$ (see \cite{shelah}).

In this paper, we introduce and study a related concept, where instead of taking cofinalities of ultraproducts, we take cuts of such ultraproducts and their cofinalities. Let $\kappa$ be a regular cardinal and let $\bar a$ be a non-decreasing sequence of regular cardinals. The set of possible cuts of $\bar a$ is defined as
\begin{center}
$\Pcut(\bar a)=\{(\lambda_1, \lambda_2):$ $\lambda_1, \lambda_2 \geq \aleph_0$ and for some ultrafilter $D$ on $\kappa$, $(\lambda_1, \lambda_2)$ is the cofinality of a cut of $\prod\limits_{i<\kappa} \mu_i / D      \}$.
\end{center}
We find a bound on $\Pcut(\bar a)$ by showing that $|\Pcut(\bar a)| \leq 2^\kappa$.
We also discuss the optimality of this result.

\section{Bounds on $\Pcut$}
Let $\kappa$ be a regular cardinal and let $\bar a=\langle \mu_i: i<\kappa \rangle$ be an increasing sequence of regular cardinals.
In this section we discuss the cardinality of the set
$\Pcut(\bar a)$. We first prove the following:
\begin{theorem}
\label{thm1}
Suppose $\kappa$ and $\bar a$ are as above. Then $|\Pcut(\bar a)| \leq 2^\kappa$. Indeed if
$2^\kappa=\aleph_\alpha,$ then $|\Pcut(\bar a)| \leq |\alpha+1|^2$.
\end{theorem}
\begin{proof}
We prove the theorem by showing that if $(\lambda_1, \lambda_2) \in \Pcut(\bar a)$, then $\lambda_1, \lambda_2 \leq 2^\kappa.$
Thus fix $(\lambda_1, \lambda_2)$ as above and let $D$ be an ultrafilter on $\kappa$ such that $(\lambda_1, \lambda_2)$
is the cofinality of a cut of $\prod\limits_{i<\kappa} \mu_i / D$.

$(*):$ $\lambda_2 \leq 2^\kappa.$
\begin{proof}
Suppose towards a contradiction that $\lambda_2 > 2^\kappa$ and let $\bar g=(g_\xi: \xi < (2^\kappa)^+              \rangle$ be a $<_D$-decreasing sequence in $\prod\limits_{i<\kappa} \mu_i$. Define the function $F: [(2^\kappa)^+]^2 \to \kappa$
by
\[
F(\xi, \zeta)=\min\{i<\kappa: g_\zeta(i) < g_\xi(i)   \}.
\]
The function $F$ is well-defined. By the Erdos-Rado partition theorem, we can find some set $X$
of cardinality $\kappa^+$ and some $i_*<\kappa$ such that for all $\xi < \zeta$ in $X$, we have $g_\zeta(i_*) < g_\xi(i_*)$.
Thus we get an infinite strictly decreasing sequence of ordinals which is absurd. Thus $\lambda_2 \leq 2^\kappa$, as requested.
\end{proof}
$(**):$ $\lambda_1 \leq 2^\kappa.$
\begin{proof}
Suppose towards a contradiction that $\lambda_1 > 2^\kappa$. Let $D$ and $(\bar f, \bar g)$ witness that $(\lambda_1, \lambda_2) \in \Pcut(\bar a)$, where $\bar f = \langle f_\alpha: \alpha < \lambda_1 \rangle$
and $\bar g= \langle g_\xi: \xi<\lambda_2 \rangle.$
For $i<\kappa$ set $U_i=\{g_\xi(i): \xi < \lambda_2       \}$. Note that $|U_i|\leq \lambda_2 \leq 2^\kappa.$ Define the relation
$E_i$ on $\lambda_1$ by $\alpha E_i \beta$ iff the following hold:
\begin{enumerate}
\item for all $\xi < \lambda_2$ and all $\delta \in U_i$,
\[
f_\alpha(\xi)< \delta \iff f_\beta(\xi) < \delta,
\]
\item for all $\xi < \lambda_2$ and all $\delta \in U_i$,
\[
f_\alpha(\xi)= \delta \iff f_\beta(\xi) = \delta.
\]
\end{enumerate}
The relation $E_i$ is clearly an equivalence relation on $\lambda_1$ and it has at most
$|U_i| \leq 2^\kappa$ many equivalence classes.
Set $E=\bigcap\limits_{i<\kappa}E_i.$ Then $E$ is also an equivalence relation and it has at most
$(2^\kappa)^\kappa=2^\kappa$ many equivalence classes.

As $\lambda_1 > 2^\kappa,$ there exists some equivalence class $Y$ which is unbounded in $\lambda_1.$
Define $h \in \prod\limits_{i<\kappa}\mu_i$ by
\[
h(i)=\sup\{f_\alpha(i): \alpha \in Y     \}
\]
Then $h$ is easily seen to be the $<_D$-least upper bound of $\bar f$
and that for all $\xi < \lambda_2, h <_D g_\xi.$ Thus
\[
\forall \alpha < \lambda_1~\forall \xi < \lambda_2 \left( f_\alpha <_D h <_D g_\xi \right)
\]
which contradicts the choice of $(\bar f, \bar g)$.
\end{proof}
The theorem  follows from $(*)$ and $(**)$.
\end{proof}
\begin{corollary}
\label{co1}
Suppose $2^\kappa$ is not a fixed point of the $\aleph$-function. Then  $|\Pcut(\bar a)| < 2^\kappa$.
\end{corollary}
\begin{proof}
Let $\alpha$ be such that $2^\kappa=\aleph_\alpha$. Then $\alpha < 2^\kappa$, and hence by Theorem \ref{thm1},
\[
|\Pcut(\bar a)| \leq |\alpha+1|^2 < 2^\kappa,
\]
as required.
\end{proof}
\begin{remark}
\label{re1}
If in the definition of the set $\Pcut(\bar a)$, we allow $\lambda_2$ to be $1$, then the set $\Pcut(\bar a)$
can be large. Indeed if $D$ is an ultrafilter on $\kappa$ and for all $i, \theta_i < \mu_i$ is a regular cardinal, then
$$(\tcf(\prod\limits_{i<\kappa}\theta_i)/D, 1) \in \Pcut,$$
provided that $\tcf(\prod\limits_{i<\kappa}\theta_i)/D$ exists.
In particular if $\mu=\lim_D \langle \mu_i: i<\kappa \rangle=\aleph_\alpha,$
then $|\alpha| \leq |\Pcut(\bar a)|$.
\end{remark}
The next lemma shows that $\Pcut$ can take its maximal possible number of elements.
\begin{lemma}
\label{thm2}
Let $\bar a= \langle  \mu_i: i<\kappa   \rangle$ where $\kappa=\aleph_0$ and for $i<\omega$, $\mu_i=\omega$. Let $\theta > \aleph_0$ be an inaccessible cardinal and set $\bbP=\Add(\omega, \theta).$ Then $V[\bold G_\bbP] \models$``$|\Pcut(\bar a)|=\theta=2^{\aleph_0}$''.
\end{lemma}
\begin{proof}
By \cite{canjar},
\[
V[\bold G_\bbP] \models\text{``~}\Pcut(\bar a) \supseteq \{(\lambda_1, \lambda_2): \lambda_1, \lambda_2 \leq \kappa \text{~are uncountable regular cardinals}    \}\text{''},
\]
from which the result follows.
\end{proof}

\end{document}